\documentclass[12pt]{article}
\usepackage[top=2cm,bottom=2cm,left=2cm,right=2cm]{geometry}
\usepackage{authblk,amsmath,amsthm,amssymb,enumerate}
\usepackage{graphicx,epstopdf,color}
\usepackage[colorlinks=true,pdfstartview={XYZ null null 1.00},pdfview=FitH,citecolor=blue]{hyperref}
\def\a{\alpha}
\def\b{\beta}
\def\d{\delta}
\def\ep{\varepsilon}
\def\G{\Gamma}
\def\l{\lambda}
\def\th{\theta}
\def\ds{\displaystyle}
\def\cut{\setminus}
\def\C{\mathbb C}
\def\tC{\widetilde C}
\def\R{\mathbb R}

\DeclareMathOperator\re{{Re}}
\DeclareMathOperator\im{{Im}}

\newtheorem{thm}{Theorem}[section]
\newtheorem{lem}[thm]{Lemma}
\numberwithin{equation}{section}

\begin{document}

\title{\bf Error bounds for the asymptotic expansions of the Hermite polynomials}

\author[1]{Wei Shi}
\author[2]{Gerg\H{o} Nemes}
\author[3]{Xiang-Sheng Wang\thanks{Corresponding author. Email: xswang@louisiana.edu}}
\author[4]{Roderick Wong}

\affil[1]{College of Science, Huazhong Agricultural University, Wuhan 430070, P. R. China}
\affil[2]{Alfr\'ed R\'enyi Institute of Mathematics, Re\'altanoda utca 13-15, Budapest H-1053, Hungary}
\affil[3]{Department of Mathematics, University of Louisiana at Lafayette, Lafayette, LA 70503, USA}
\affil[4]{Department of Mathematics, City University of Hong Kong, Tat Chee Avenue, Kowloon, Hong Kong}

\date{}

\maketitle
\begin{abstract}
  In this paper, we present explicit and computable error bounds for the asymptotic expansions of the Hermite polynomials with Plancherel--Rotach scale. Three cases, depending on whether the scaled variable lies in the outer or oscillatory interval, or it is the turning point, are considered separately. We introduce the ``branch cut" technique to express the error terms as integrals on the contour taken as the one-sided limit of curves approaching the branch cut. This new technique enables us to derive simple error bounds in terms of elementary functions. We also provide recursive procedures for the computation of the coefficients appearing in the asymptotic expansions.
\end{abstract}



\section{Introduction}
Plancherel--Rotach asymptotic expansions for the orthogonal polynomials as the polynomial degree tends to infinity have been studied  extensively in the literature. There are many standard asymptotic techniques including the steepest descent method for integrals \cite{Wong89}, the WKB method for differential equations \cite{Olver97}, the Deift--Zhou method for Riemann--Hilbert problems \cite{Deift99CPAM,Deift93AM}, asymptotic theory of difference equations \cite{Wong14AA}, and Darboux's method \cite{Wong05CA}. The error term for the truncated asymptotic expansion is usually of the same order of magnitude as the first neglected term. In other words, the ratio of the error term and the first omitted term is bounded by a constant independent of the polynomial degree. The existence of such a bound can be proved by standard ``soft" analysis. However, as far as we know, the quantitative information for this error bound is unknown even for the classical orthogonal polynomials. One may expect a rather complicated ``hard" analysis on finding explicit and computable expressions for error bounds.
Various reasons why we need error bounds for asymptotic expansions are discussed in \cite{Olver80SIREV}.
As mentioned in \cite{Wong80SIREV}, the upper bounds of the error term obtained from standard asymptotic techniques are difficult to compute, and thus may not be realistic. In this paper, we will derive explicit and computable error bounds for the asymptotic expansions of the Hermite polynomials.

The main challenge is to find a convenient integral representation of the error term on an appropriate contour so that the error bound is computable. Berry and Howls \cite{Berry91PRSA} proposed the so-called ``adjacent saddles" method which was further developed in \cite{Bennett2018,Boyd93PRSA}.
The key idea of their method is to express the error term as an integral over the ``adjacent contours" passing through the ``adjacent saddles".
However, the technique of ``adjacent saddles" cannot be applied directly to the Hermite polynomials and many other orthogonal polynomials because the integrand (or the phase function) of the integral representation for the Hermite polynomials has a branch point in the complex plane.
It is thus difficult to express the error term as an integral over the ``adjacent contours". Moreover, for the turning point case, when two saddle points coincide with each other, there does not exist any ``adjacent saddle". To resolve these two difficulties, we introduce a new ``branch cut" technique which deforms the contour of integration for the error term to the branch cut of the phase function. More precisely, the contour is defined as the limit of the curves approaching one side of the branch cut. By virtue of the ``branch cut" technique, we are able to find simple error bounds in terms of elementary functions.

The Hermite polynomials can be expressed as contour integrals:
\begin{equation*}
  H_n(x)={n!\over2\pi i}\int_\G e^{2xt-t^2}t^{-n-1}dt,
\end{equation*}
where $\G$ is a counter-clockwisely oriented contour encircling the negative real line (cf. \cite[\href{http://dlmf.nist.gov/18.10.iii}{\S 18.10(iii)}]{NIST:DLMF}). To be more specific, we may choose
\begin{equation*}
  \G=\{\l-i\d:\l<0\}\cup\{\d e^{i\th}:-\pi/2\le\th\le\pi/2\}\cup\{\l+i\d:\l<0\},
\end{equation*}
where $\d>0$ is any fixed positive number; see Figure \ref{fig-G}.

\begin{figure}[htp]
  \centering
  \includegraphics[height=0.08\textheight]{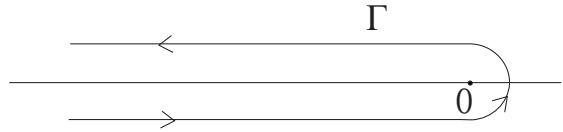}
  \caption{The contour $\G$ with counter-clockwise orientation.}
  \label{fig-G}
\end{figure}

After introducing the Plancherel--Rotach scale, we rewrite the integral representation as
\begin{equation}\label{Hn-G}
  H_n(\sqrt Nx)
  = {n!\over2\pi iN^{n/2}}\int_\G e^{-N[t^2-2xt+(\ln t)/2]}t^{-1/2}dt,
\end{equation}
where $N=2n+1$.
By symmetry, we may assume that $x\ge0$.
We shall denote the phase function in the above integral by
\begin{equation}\label{f}
  f(t;x)=t^2-2xt+{1\over2}\ln t.
\end{equation}
The zeros of $f'(t;x)$ are called saddle points and can be calculated as
\begin{align*}
  t_\pm={x\pm\sqrt{x^2-1}\over2}=
  \begin{cases}
    e^{\pm\b}/2,&~x=\cosh\b>1,~\b>0;\\
    e^{\pm i\a}/2,&~x=\cos\a\in[0,1),~\a\in(0,\pi/2];\\
    1/2,&~x=1.
  \end{cases}
\end{align*}
Note that the two saddle points coincide ($t_+=t_-$) at the turning point $x=1$.
For simplicity, we shall drop the dependence of $f$ on $x$.
There are three cases to be considered.
\begin{enumerate}[{Case} I:]
  \item $x=\cosh\b>1$ with $\b>0$. This is called the outer interval.
  \item $x=\cos\a\in[0,1)$ with $\a\in(0,\pi/2]$. This is called the oscillatory interval.
  \item $x=1$. This is called the turning point.
\end{enumerate}
We will consider these three cases separately. We mention here that an asymptotic expansion with error bounds for the slightly differently scaled $H_n(\sqrt{N+1}\cos\a)$ was given earlier by van Veen \cite{vanVeen31MA}.

The rest of the paper is organized as follows. In Sections \ref{sec2}--\ref{sec4}, we derive the asymptotic expansions with error bounds for Case I, II and III, respectively. The main results are stated at the end of each section. In Section \ref{sec5}, we demonstrate the accuracy of the error bounds through numerical examples. Recursive procedures for the computation of the coefficients appearing in the asymptotic expansions are given in Appendix \ref{appendixa}.

\section{Case I: \texorpdfstring{$x=\cosh\b>1$}{x = cosh beta > 1} with \texorpdfstring{$\b>0$}{beta > 0}}\label{sec2}
The saddle points are $t_\pm=e^{\pm\b}/2$ with $0<t_-<t_+$.
Moreover, $f''(t_-)=2(1-e^{2\b})<0$ and $f''(t_+)=2(1-e^{-2\b})>0$.
We shall deform the contour of integration $\G$ to the path of steepest descent passing through the saddle point $t_-$.
To describe this steepest descent contour, we shall introduce the analytic function
\begin{equation}\label{w}
  w(z)=[f(z)-f(t_-)]^{1/2},
\end{equation}
where the branch of the square root function is chosen so that
\[
  w(z)=-i(z-t_-)\sqrt{-f''(t_-)/2}+\mathcal{O}((z-t_-)^2)=-i(z-t_-)\sqrt{e^{2\b}-1}+\mathcal{O}((z-t_-)^2),
\]
for $z$ in a small complex neighborhood of $t_-$.
Let $z=re^{i\th}$ with $r>0$ and $\th\in(-\pi,\pi)$. We investigate the equation $\im[f(z)-f(t_-)] = \im f(z)=0$; namely,
\begin{equation}\label{r-th}
  r^2\sin(2\th)-2 r\cosh \b \sin\th+{\th\over2}=0.
\end{equation}
Clearly, $\th=0$ satisfies the above equation. By symmetry, we can restrict $\th\in(0,\pi)$ and find two solutions
\begin{align}
  r_-(\th)&={\th/(2\sin\th)\over \cosh \b+\sqrt{\cosh^2 \b-\th/\tan\th}},~~\th\in(0,\pi),\nonumber \\
  r_+(\th)&={\cosh \b+\sqrt{\cosh^2 \b-\th/\tan\th}\over2\cos\th},~~\th\in(0,\pi/2).\label{r+}
\end{align}
It is readily seen that $r_\pm(\th)\to t_\pm=e^{\pm\b}/2$ as $\th\to0$. For consistency, we define $r_\pm(0)=t_\pm$.
Consequently, we can regard $r_-(\th)$ as a symmetric function on $(-\pi,\pi)$ and $r_+(\th)$ a symmetric function on $(-\pi/2,\pi/2)$.
Define two contours
\begin{align*}
  \G_- = \{r_-(\th)e^{i\th}:\th\in(-\pi,\pi)\},~~
  \G_+ = \{r_+(\th)e^{i\th}:\th\in(-\pi/2,\pi/2)\}.
\end{align*}
It follows that the solution of $\im f(z)=0$ is the union $\G_-\cup\G_+\cup\R^+$, where $\R^+$ is the positive real line.
The contours $\G_\pm$ are depicted in Figure \ref{fig-case1}.

\begin{figure}[htp]
  \centering
  \includegraphics[height=7cm]{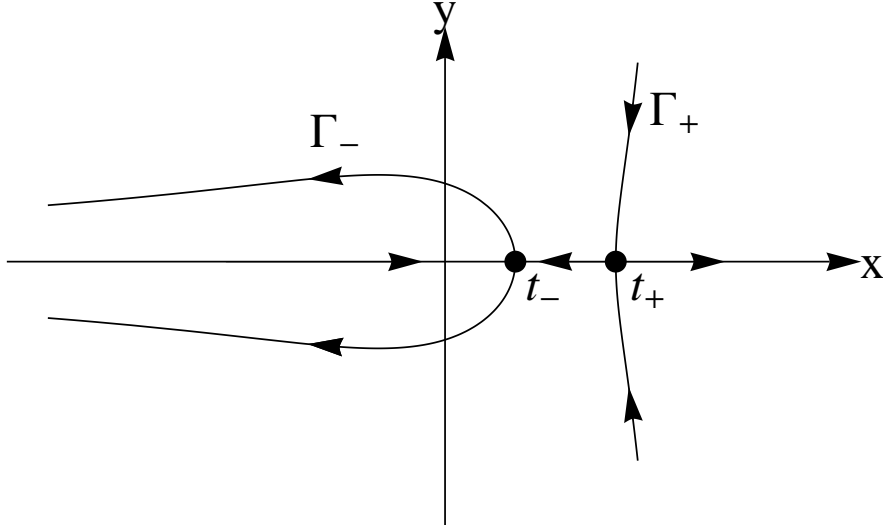}
  \caption{The contours $\G_\pm$. The arrows indicate the increasing direction of the phase function $f$.}
  \label{fig-case1}
\end{figure}

Since $f'(z)$ has no zeros other than $t_\pm$, it follows from the asymptotic behavior $f(z)\sim z^2$ as $z\to\infty$
that $\G_-$
is a steepest descent contour passing through $t_-$, while
$\G_+$
is a steepest ascent contour passing through $t_+$.
Moreover, the solutions of $f(z)-f(t_-)=0$ are real and positive.
Consider $f(t)=t^2-2t\cosh\b+(\ln t)/2$ for $t>0$. Since $f(t)\to-\infty$ as $t\to0^+$ and $f(t)\to\infty$ as $t\to\infty$, we obtain from $f'(t_\pm)=0$ that
the equation $f(t)-f(t_-)=0$ has exactly two positive roots: one is $t_-$; and the other, denoted by $t_c$, is larger than $t_+$.
Hence, we can choose $[t_c,\infty)$ as a branch cut and extend the definition of $w(z)$ on $\C\cut(-\infty,0]\cup[t_c,\infty)$ by analytic continuation.

Now, we deform the contour of integration from $\G$ to the steepest descent contour $\G_-$ with counter-clockwise orientation and then make a change of variable $u=w(t)$:
\begin{equation}\label{Hn-G-}
  H_n(\sqrt N\cosh \b)={n!e^{-Nf(t_-)}\over2\pi iN^{n/2}}\int_{\G_-} e^{-N[f(t)-f(t_-)]}t^{-1/2}dt
  ={n!e^{-Nf(t_-)}\over2\pi iN^{n/2}}\int_\R e^{-Nu^2}{[w^{-1}(u)]^{-1/2}\over w'(w^{-1}(u))}du.
\end{equation}
Recall that $\G_+$ is the steepest ascent contour passing through $t_+$.
We also denote by $\G_0$ a clockwise oriented contour encircling $(-\infty,0]$; namely,
\begin{equation*}
  \G_0=\{\l+i\d:\l<0\}\cup\{\d e^{i\th}:-\pi/2\le\th\le\pi/2\}\cup\{\l-i\d:\l<0\},
\end{equation*}
where $\d>0$ is a small positive number such that $\d<t_c$.
In fact, $\G_0$ is identical to $\G$ but has the opposite orientation.
Since $|w(z)|/|z|\to1$ as $z\to\infty$, we have from the Cauchy integral formula
\begin{equation*}
  {1\over2\pi i}\int_{\G_0\cup\G_+}{z^{-1/2}\over w(z)-u}dz={[w^{-1}(u)]^{-1/2}\over w'(w^{-1}(u))}
\end{equation*}
for any $u\in\R$. Substituting this expression into \eqref{Hn-G-} gives a double integral representation
\begin{equation*}
  H_n(\sqrt N\cosh \b)={n!e^{-Nf(t_-)}\over(2\pi i)^2N^{n/2}}\int_\R e^{-Nu^2}\int_{\G_0\cup\G_+}{z^{-1/2}\over w(z)-u}dzdu.
\end{equation*}
For any $p\ge0$, it follows from
\begin{equation*}
  {1\over w(z)-u}=\sum_{k=1}^{2p}{u^{k-1}\over w(z)^k}+{u^{2p}\over w(z)^{2p}[w(z)-u]}
\end{equation*}
and
\begin{equation}\label{int-k}
  \int_\R e^{-Nu^2}u^{k-1}du=\begin{cases}
    0,&~k=2j,\\
    \G(j+1/2)/N^{j+1/2},&~k=2j+1
  \end{cases}
\end{equation}
that
\begin{gather}\label{Hn-epp}
\begin{split}
  H_n(\sqrt N\cosh \b) =&\; {n!e^{-Nf(t_-)}\over(2\pi i)^2N^{n/2}}
  \bigg\{\sum_{k=1}^{2p}\int_\R e^{-Nu^2}u^{k-1}du\int_{\G_0\cup\G_+}{z^{-1/2}\over w(z)^k}dz
  \\& +\int_\R e^{-Nu^2}u^{2p}\int_{\G_0\cup\G_+}{z^{-1/2}\over w(z)^{2p}[w(z)-u]}dzdu\bigg\}
  \\  =&\;  \frac{n!e^{-Nf(t_-)}}{2\sqrt{\pi\sinh\b}N^{(n+1)/2}}\left\{\sum_{j=0}^{p-1}\frac{A_j(\coth \b)}{N^j}  + \ep_p(N,\b)\right\},
\end{split}
\end{gather}
where
\begin{align}
  \label{firstAint} A_j(\coth \b) =& -\frac{\sqrt{\sinh\b}\G(j+1/2)}{2\pi^{3/2}}\int_{\G_0\cup\G_+}{z^{-1/2}\over w(z)^{2j+1}}dz, \\
   \ep_p(N,\b)=& -\frac{\sqrt{N\sinh\b}}{2\pi^{3/2}}\int_\R e^{-Nu^2}u^{2p}\int_{\G_0\cup\G_{+}}{z^{-1/2}\over w(z)^{2p}[w(z)-u]}dzdu. \label{epp}
\end{align}
We show in Appendix \ref{appendixa1} that the coefficients $A_j(\coth \b)$ are polynomials in $\coth \b$ of degree $3j$ with rational coefficients, and can be calculated using a recursive formula.

To find an upper bound for the remainder $\ep_p(N,\b)$, we need to estimate the integral
\begin{equation}\label{cpu}
  c_p(u)=\int_{\G_0\cup\G_+}{z^{-1/2}\over w(z)^{2p}[w(z)-u]}dz,~~u\in\R.
\end{equation}
If there exists $C_p>0$ such that $|c_p(u)|\le C_p$ for all $u\in\R$, then we obtain from \eqref{int-k} with $k=2p+1$ and \eqref{epp} that
\begin{equation}\label{epp-bound}
  |\ep_p(N,\b)|\le \frac{\sqrt{\sinh\b}}{2\pi^{3/2}} \frac{C_p\G(p+1/2)}{N^p}.
\end{equation}
First, we estimate the integral
\begin{equation*}
  I_0=\int_{\G_0}{z^{-1/2}\over w(z)^{2p}[w(z)-u]}dz.
\end{equation*}
Recall that the contour $\G_0$ encircles the negative real line having a distance $\d>0$ from it.
By analyticity of the integrand, the value $I_0$ does not change if we let $\d\to0^+$. Hence,
\begin{equation}\label{I0}
  |I_0|\le \int_0^{+\infty} {s^{-1/2}\over |w(se^{i\pi})|^{2p}|w(se^{i\pi})-u|}ds
  +\int_0^{+\infty} {s^{-1/2}\over |w(se^{-i\pi})|^{2p}|w(se^{-i\pi})-u|}ds,
\end{equation}
where $w(se^{\pm i\pi})$ is defined as the limit of $w(se^{i\th})$ as $\th\to\pm\pi$.
Note from \eqref{f} and \eqref{w} that
\begin{equation*}
  w(se^{\pm i\pi})^2=s^2+2s\cosh \b+(\ln s)/2-f(t_-)\pm i\pi/2.
\end{equation*}
Thus, we obtain
\begin{equation}\label{R}
  |w(se^{\pm i\pi})|=\{[s^2+2s\cosh \b +(\ln s)/2-f(t_-)]^2+(\pi/2)^2\}^{1/4}=:R(s),
\end{equation}
where the last equality defines $R(s)$, the modulus of $w(se^{\pm i\pi})$.
Let $\th_\pm(s)$ be the phase of $w(se^{\pm i\pi})$.
We then have
\begin{equation*}
  R(s)^2\cos[2\th_\pm(s)]=\re[w(se^{\pm i\pi})^2]=s^2+2s\cosh \b +(\ln s)/2-f(t_-),
\end{equation*}
which implies that
\begin{align*}
  &|w(se^{\pm i\pi})-u|\ge|\im[w(se^{\pm i\pi})]|=R(s)|\sin[\th_\pm(s)]|
  =R(s)\sqrt{1-\cos[2\th_\pm(s)]\over2}
  \notag\\&=\sqrt{R(s)^2-[s^2+2s\cosh \b +(\ln s)/2-f(t_-)]\over2}
  =\sqrt{(\pi/2)^2/2\over R(s)^2+[s^2+2s\cosh \b +(\ln s)/2-f(t_-)]}
  \notag\\&\ge\sqrt{(\pi/2)^2\over 4R(s)^2}={\pi\over4R(s)}.
\end{align*}
Assume that $p\geq 1$. Substituting this inequality into \eqref{I0} yields
\begin{align*}
  |I_0|\le {8\over\pi}\int_0^{+\infty}{s^{-1/2}\over R(s)^{2p-1}}ds.
\end{align*}
To estimate the integral on the right-hand side of the above inequality, we note from \eqref{R} that
$R(s)\ge\sqrt{\pi/2}$ for $s\in[0,1]$ and
$R(s)\ge\sqrt{2s\cosh \b }$ for $s\ge1>t_-$. Thus,
\begin{align}\label{I0-bound}
  |I_0|\le {8\over\pi}\int_0^1{s^{-1/2}\over(\pi/2)^{p-1/2}}ds+{8\over\pi}\int_1^{+\infty}{s^{-1/2}\over(2s\cosh \b)^{p-1/2}}ds
  ={2^{p+7/2}\over\pi^{p+1/2}}+{1\over\pi(p-1)2^{p-7/2}(\cosh \b)^{p-1/2}},
\end{align}
where we have assumed $p\ge2$ to ensure convergence of the second integral.

It remains to estimate the integral
\begin{equation*}
  I_c=\int_{\G_+}{z^{-1/2}\over w(z)^{2p}[w(z)-u]}dz.
\end{equation*}
For $z\in\G_+$, we have $\im f(z)=0$ and $\re [f(z)-f(t_-)]<f(t_+)-f(t_-)<0$.
It then follows from \eqref{w} that $w(z)$ is purely imaginary and $|w(z)-u|\ge|w(z)|$ for $z\in\G_+$ and $u\in\R$.
Now, we parameterize $\G_+$ as $z=z_+(\th)=r_+(\th)e^{i\th}$ with $\th\in(-\pi/2,\pi/2)$, where $r_+(\th)$ is given in \eqref{r+}.
Since $dz/d\th=[r_+'(\th)+ir_+(\th)]e^{i\th}$ and
$|w(z_+(\th))|^2=|f(z_+(\th))-f(t_-)|=f(t_-)-f(z_+(\th))$,
we deduce
\begin{align*}
  |I_c|\le&\int_{-\pi/2}^{\pi/2}{r_+(\th)^{-1/2}\sqrt{[r_+'(\th)]^2+[r_+(\th)]^2}\over
  |w(z_+(\th))|^{2p+1}}d\th
  \notag\\=&\int_0^{\pi/2}{2r_+(\th)^{-1/2}\sqrt{[r_+'(\th)]^2+[r_+(\th)]^2}\over
  [f(t_-)-f(z_+(\th))]^{p+1/2}}d\th.
\end{align*}
As both $\th/\tan\th$ and $\cos\th$ are decreasing functions for $\th\in(0,\pi/2)$, it follows from \eqref{r+} that $r_+(\th)$ is an increasing function for $\th\in(0,\pi/2)$; namely, $r_+'(\th)\ge0$.
To estimate $r_+'(\th)$, we apply an implicit differentiation on \eqref{r-th} and obtain
\begin{align*}
  r_+'(\th)={-2r_+(\th)^2\cos(2\th)+2r_+(\th) \cosh\b \cos\th-1/2\over2r_+(\th)\sin(2\th)-2\cosh\b\sin\th}.
\end{align*}
A simple calculation together with \eqref{r+} gives
\[
  2r_+(\th)\sin(2\th)-2\cosh\b\sin\th= 2\sin\th[2r_+(\th)\cos\th-\cosh\b]=2\sin\th\sqrt{\cosh^2\b-\th/\tan\th}
\]
and
\begin{gather}
\begin{split}
 & 2 r_+(\th)\cosh\b\cos\th-2r_+(\th)^2\cos^2\th \\ & =\cosh\b(\cosh\b+\sqrt{\cosh^2\b-\th/\tan\th})-{(\cosh\b+\sqrt{\cosh^2\b-\th/\tan\th})^2\over2}={\th\over2\tan\th}.\label{r-cos}
\end{split}
\end{gather}
Consequently,
\begin{align*}
  0\le r_+'(\th)={2r_+(\th)^2\sin^2\th+\th/(2\tan\th)-1/2\over2\sin\th\sqrt{\cosh^2\b-\th/\tan\th}}
  \le{r_+(\th)^2\sin\th\over\sqrt{\cosh^2\b-\th/\tan\th}}\le\frac{r_+(\th)\tan\th}{\tanh \b},
\end{align*}
and
\begin{align}\label{Ic0}
  |I_c|\le\frac{2}{\tanh \b}\int_0^{\pi/2}{r_+(\th)^{1/2}/\cos\th\over
  [f(t_-)-f(z_+(\th))]^{p+1/2}}d\th.
\end{align}
Let $\th_0=\arccos(1/4)$. If $\th\in(0,\th_0)$, then $\cos\th>1/4$, $f(t_-)-f(z_+(\th))>f(t_-)-f(t_+)=\sinh(2\b)/2-\b$ and $r_+(\th)<4\cosh\b$. Hence,
\begin{align}\label{Ic1}
  \int_0^{\th_0}{r_+(\th)^{1/2}/\cos\th\over
  [f(t_-)-f(z_+(\th))]^{p+1/2}}d\th\le{2\pi(4\cosh \b)^{1/2}\over[\sinh(2\b)/2-\b]^{p+1/2}}.
\end{align}
If $\th\in(\th_0,\pi/2)$, then $r_+(\th)>2e^{\b}=4t_+>2\cosh \b>2$. This together with \eqref{r-cos} implies that
\begin{align*}
  -f(z_+(\th))& =-r_+(\th)^2\cos(2\th)+2r_+(\th)\cosh \b \cos\th-{1\over2}\ln r_+(\th)
=r_+(\th)^2+{\th\over2\tan\th}-{1\over2}\ln r_+(\th)
  \notag\\& \ge{4\over5}r_+(\th)^2+{1\over5}r_+(\th)^2-{1\over2}\ln r_+(\th)
  \ge{4\over5}r_+(\th)^2,
\end{align*}
and
\begin{align*}
  -f(t_-)={e^{-2\b}\over4}+{\b\over2}+{1+\ln 2\over2}\le{e^{2\b}\over4}+{19\over20}\le t_+^2+{19\over80}r_+(\th)^2\le{3\over10}r_+(\th)^2.
\end{align*}
Coupling the above two inequalities gives
\begin{equation*}
  f(t_-)-f(z_+(\th))\ge{1\over2}r_+(\th)^2.
\end{equation*}
Therefore,
\begin{align}\label{Ic2}
  &\int_{\th_0}^{\pi/2}{r_+(\th)^{1/2}/\cos\th\over
  [f(t_-)-f(z_+(\th))]^{p+1/2}}d\th\le
  \int_{\th_0}^{\pi/2}{r_+(\th)^{1/2}\cdot r_+(\th)/t_+\over
  r_+(\th)^{2p+1}/2^{p+1/2}}d\th
  \notag\\& ={2^{p+1/2}\over t_+}\int_{\th_0}^{\pi/2}r_+(\th)^{-2p+1/2}d\th
  \le{2^{p+1/2}\over t_+}\int_{\pi/3}^{\pi/2}(4t_+)^{-2p+1/2}d\th
  ={2^{-3p+1/2}\pi\over3 t_+^{2p+1/2}},
\end{align}
where we have assumed $p\ge1$. Substituting \eqref{Ic1} and \eqref{Ic2} into \eqref{Ic0} yields
\begin{align}\label{Ic-bound}
  |I_c|\le {2\over\tanh\b}\bigg\{{4\pi \sqrt{\cosh\b}\over[\sinh(2\b)/2-\b]^{p+1/2}}+{2^{-p+1}\pi\over3 e^{(2p+1/2)\b}}\bigg\},~~p\ge1.
\end{align}
A combination of \eqref{cpu}, \eqref{I0-bound} and \eqref{Ic-bound} gives
\begin{align*}
  |c_p(u)|\le|I_0|+|I_c|\le&\; {2^{p+7/2}\over\pi^{p+1/2}}+{1\over\pi(p-1)2^{p-7/2}(\cosh\b)^{p-1/2}}
  \notag\\&+{2\over\tanh\b}\bigg\{{4\pi\sqrt{\cosh\b}\over[\sinh(2\b)/2-\b]^{p+1/2}}+{2^{-p+1}\pi\over3 e^{(2p+1/2)\b}}\bigg\},
\end{align*}
for $p\ge2$. Let $C_p$ be the number on the right-hand side of the above inequality. The error bound \eqref{epp-bound} holds for $p\ge2$.
Note from \eqref{Hn-epp} that
\begin{equation*}
  \ep_p(N,\b)={A_p(\coth \b) \over N^p}+\ep_{p+1}(N,\b).
\end{equation*}
Hence, we may improve the error bound \eqref{epp-bound} to
\begin{equation*}
  |\ep_p(N,\b)|\le {|A_p(\coth \b)| \over N^p}+\frac{\sqrt{\sinh\b}}{2\pi^{3/2}} {C_{p+1}\G(p+3/2)\over N^{p+1}}={\tC_p \over N^p},
\end{equation*}
for $p\ge1$, where
\begin{equation*}
  \tC_p=|A_p(\coth \b)|+\frac{\sqrt{\sinh\b}}{2\pi^{3/2}} {C_{p+1}\G(p+3/2)\over N}.
\end{equation*}
Using the explicit value
\[
-f(t_-) = \frac{e^{-2\b}}{4}+\frac{\b}{2}+\frac{1+\ln 2}{2}
\]
in \eqref{Hn-epp}, we can summarize our results as follows.

\begin{thm}\label{thm-case1} Let $\b >0$ and $N=2n+1 \geq 1$. Then for any $p\ge1$, we have
\begin{equation}\label{Hnasympt}
 H_n(\sqrt N\cosh \b)=\frac{2^n n!e^{N(e^{-2\b}+2\b+2)/4}}{\sqrt{2\pi\sinh\b}N^{(n+1)/2}} \left\{ \sum_{j=0}^{p-1}\frac{A_j(\coth \b)}{N^j}  + \ep_p(N,\b)\right\},
\end{equation}
where
\begin{equation}\label{bound1}
  |\ep_p(N,\b)|\le {\tC_p \over N^p},
\end{equation}
with
\begin{equation*}
  \tC_p=|A_p(\coth \b)|+\frac{\sqrt{\sinh\b}}{2\pi^{3/2}} {C_{p+1}\G(p+3/2)\over N},
\end{equation*}
and
\begin{align*}
C_{p+1}=&{2^{p+9/2}\over\pi^{p+3/2}}+{1\over\pi p2^{p-5/2}(\cosh\b)^{p+1/2}}
  +{2\over\tanh\b}\bigg\{{4\pi\sqrt{\cosh\b}\over[\sinh(2\b)/2-\b]^{p+3/2}}+{2^{-p}\pi\over3 e^{(2p+5/2)\b}}\bigg\}.
\end{align*}
In particular, $\tC_p\to |A_p(\coth \b)|$ as $N\to+\infty$; namely, the error bound is close to the absolute value of the first neglected term when $N$ is large. The coefficients $A_j(\coth \b)$ are polynomials in $\coth \b$ of degree $3j$ with rational coefficients, and these polynomials can be calculated using a recursive formula given in Appendix \ref{appendixa1}.
\end{thm}

\section{Case II: \texorpdfstring{$x=\cos\a\in[0,1)$}{x = cos alpha in [0,1)} with \texorpdfstring{$\a\in(0,\pi/2]$}{alpha in (0,pi/2]}}
The saddle points are $t_\pm=e^{\pm i\a}/2$, and we have
\[
 f''(t_\pm)=2-2e^{\mp2i\a}=4\sin\a e^{\pm i(\pi/2-\a)}.
\]
Now, we define an analytic function $w(z)$ via
\begin{equation}\label{w+-}
w(z)=[f(z)-f(t_+)]^{1/2},
\end{equation}
where the branch of the square root function is chosen so that
\[
  w(z)=\sqrt{2\sin\a}e^{i(\pi/4-\a/2)}(z-t_+)+\mathcal{O}((z-t_+)^2)
\]
for $z$ in a small complex neighborhood of $t_+$. The steepest descent contour passing through $t_+$ is described via the equation $\im[f(z)-f(t_+)]=0$. By using polar coordinates $z=re^{i\th}$, this equation becomes
\[
  r^2\sin(2\th)-2r\cos\a\sin\th+{\th-\th_0\over2}=0,
\]
with
\[
  \th_0=\a-\sin(2\a)/2\in(0,\a).
\]
We are only interested in solutions in the upper half-plane; namely, $\th\in(0,\pi)$.
By the quadratic formula, there are exactly two solutions:
\begin{align*}
  r_{d+}(\th)=\begin{cases}
    \ds{(\th-\th_0)/(2\sin\th)\over\cos\a+\sqrt{\cos^2\a-(\th-\th_0)/\tan\th}},&~\th\in[\a,\pi),\\\\
    \ds{\cos\a+\sqrt{\cos^2\a-(\th-\th_0)/\tan\th}\over2\cos\th},&~\th\in(0,\a],
  \end{cases}
\end{align*}
and
\begin{align*}
  r_{a+}(\th)=\begin{cases}
    \ds{\cos\a+\sqrt{\cos^2\a-(\th-\th_0)/\tan\th}\over2\cos\th},&~\th\in[\a,\pi/2),\\\\
    \ds{\cos\a-\sqrt{\cos^2\a-(\th-\th_0)/\tan\th}\over2\cos\th},&~\th\in(\th_0,\a],
  \end{cases}
\end{align*}
where $r_{d+}(\th)$ and $r_{a+}(\th)$ are defined as piecewise functions so that they are differentiable at $\th=\a$.
This is because the function $\sqrt{\cos^2\a-(\th-\th_0)/\tan\th}$ is not differentiable at $\th=\a$.
For any $\th$ in a small real neighborhood of $\a$, we have, by Taylor expansion, $\cos^2\a-(\th-\th_0)/\tan\th=(\th-\a)^2+\mathcal{O}((\th-\a)^3)$ and $\sqrt{\cos^2\a-(\th-\th_0)/\tan\th}=|\th-\a|+\mathcal{O}((\th-\a)^2)$. Hence,
\begin{align*}
  r_{d+}(\th)&={\cos\a-(\th-\a)\over2[\cos\a-(\th-\a)\sin\a]}+\mathcal{O}((\th-\a)^2)
  ={1\over2}-{1-\sin\a\over2\cos\a}(\th-\a)+\mathcal{O}((\th-\a)^2),\\
  r_{a+}(\th)&={\cos\a+(\th-\a)\over2[\cos\a-(\th-\a)\sin\a]}+\mathcal{O}((\th-\a)^2)
  ={1\over2}+{1+\sin\a\over2\cos\a}(\th-\a)+\mathcal{O}((\th-\a)^2).
\end{align*}
In particular, we have
\[
  r_{d+}(\a)=r_{a+}(\a)={1\over2},~~r_{d+}'(\a)=-{1\over2}\tan\left({\pi\over4}-{\a\over2}\right),~~r_{a+}'(\a)={1\over2}\cot\left({\pi\over4}-{\a\over2}\right).
\]
Now, we define two contours
\begin{align*}
  \G_{d+}=\{r_{d+}(\th)e^{i\th}:\th\in(0,\pi)\},~~
  \G_{a+}=\{r_{a+}(\th)e^{i\th}:\th\in(\th_0,\pi/2)\}.
\end{align*}
Recall that $-Nf(t)$ is the phase function in the integral representation of $H_n(\sqrt N \cos \a)$.
It is readily seen from the asymptotic behaviors $f(z)\sim z^2$ as $z\to\infty$ and $2f(z)\sim \ln z$ as $z\to0$
that $\G_{d+}$ is a steepest descent contour and $\G_{a+}$ is a steepest ascent contour passing through the saddle point $t_+=e^{i\a}/2$.
Moreover, the only zero of $f(z)-f(t_+)$ in the upper half-plane is $t_+$.
Hence, $w(z)$ defined in \eqref{w+-} is analytic for all $\im z>0$ and positive $z>0$.
For $z=se^{i\pi}$ with $s>0$, we define
\begin{equation*}
  w(se^{i\pi})=\lim_{\th\to\pi^-}w(se^{i\th}).
\end{equation*}
We can define in a similar manner the contours $\G_{d-}$ and $\G_{a-}$ which, together with $\G_{d+}$ and $\G_{a+}$, are illustrated in Figure \ref{fig-case2}.

\begin{figure}[htp]
  \centering
  \includegraphics[height=7cm]{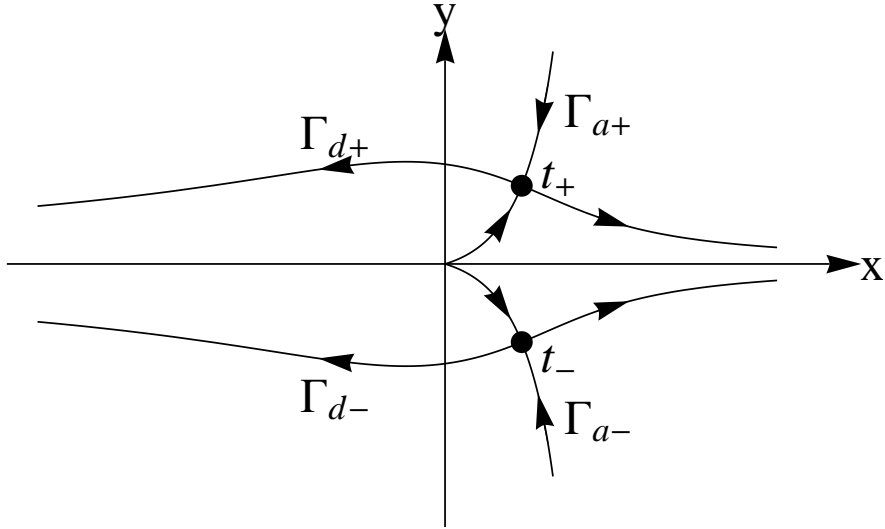}
  \caption{The contours $\G_{d\pm}$ and $\G_{a\pm}$. The arrows indicate the increasing direction of the phase function $f$.}
  \label{fig-case2}
\end{figure}

Now, we return to the integral representation of $H_n(\sqrt N \cos \a)$ in \eqref{Hn-G} and deform the contour of integration to the steepest descent contours $\G_{d+}\cup\G_{d-}$ oriented with increasing parameter $\th$ (i.e., counter-clockwise direction):
\begin{equation*}
  H_n(\sqrt N \cos\a)={n!\over2\pi iN^{n/2}}\int_{\G_{d+}} e^{-Nf(t)}t^{-1/2}dt
  +{n!\over2\pi iN^{n/2}}\int_{\G_{d-}} e^{-Nf(t)}t^{-1/2}dt=: I_+ +I_-,
\end{equation*}
where the last equality defines the two integrals $I_+$ and $I_-$, respectively. Since the contours $\G_{d+}$ and $\G_{d-}$ are symmetric with respect to the real axis, have opposite orientations, and $f(t) = \overline{f(\bar t)}$, the integrals $I_+$ and $I_-$ are complex conjugates of each other. Therefore, it suffices to study the integral $I_+$. First, we introduce the change of variable $u=w(t)$:
\begin{align*}
  I_+=-{n!e^{-Nf(t_+)}\over2\pi iN^{n/2}}\int_{\R} e^{-Nu^2}{[w^{-1}(u)]^{-1/2}\over w'(w^{-1}(u))}du.
\end{align*}
Since $w(z)$ is analytic in the upper half-plane and $|w(z)|/|z|\to1$ as $z\to\infty$, we can infer from the Cauchy integral formula that
\begin{equation*}
  {[w^{-1}(u)]^{-1/2}\over w'(w^{-1}(u))}={1\over2\pi i}\int_{\G_0^+\cup\R^+}{z^{-1/2}\over w(z)-u}dz,
\end{equation*}
where $\R^+$ is the positive real line and $\G_0^+$ is the negative real line with argument $\pi$:
\[
  \G_0^+=\lim_{\d\to0^+}\{\l+i\d:\l<0\}=\{se^{i\pi}:s>0\}.
\]
Both $\R^+$ and $\G_0^+$ are oriented from left to right.
Now, we have the following double integral representation:
\begin{equation*}
  I_+=-{n!e^{-Nf(t_+)}\over(2\pi i)^2N^{n/2}}\int_{\R} e^{-Nu^2}\int_{\G_0^+\cup\R^+}{z^{-1/2}\over w(z)-u}dzdu.
\end{equation*}
For any $p\ge0$, we use the identities \eqref{int-k} and
\begin{equation*}
  {1\over w(z)-u}=\sum_{k=1}^{2p}{u^{k-1}\over w(z)^k}+{u^{2p}\over w(z)^{2p}[w(z)-u]}
\end{equation*}
to derive
\begin{equation}\label{Inexpansion}
I_+=\frac{n!e^{-Nf(t_+)}}{2\sqrt{\pi\sin\a}N^{(n+1)/2}}\left\{\sum\limits_{j=0}^{p-1}\frac{A_j(i\cot\a)}{N^j}e^{\pi i/4}+\widetilde{\ep}_p(N,\a)\right\},
\end{equation}
where
\begin{align}
 \label{Aintegral} A_j(i\cot\a)& =e^{-\pi i/4}\frac{\sqrt{\sin\a}\G(j + 1/2)}{2\pi^{3/2}}\int_{\G_0^+\cup\R^+}{z^{-1/2}\over w(z)^{2j+1}}dz,\\
  \widetilde{\ep}_p(N,\a)  & =\frac{\sqrt{N\sin \a}}{2\pi^{3/2}}\int_\R e^{-Nu^2}u^{2p}\int_{\G_0^+\cup\R^+}{z^{-1/2}\over w(z)^{2p}[w(z)-u]}dzdu. \label{epp+}
\end{align}
The coefficients $A_j(i\cot\a)$ are polynomials in $i\cot\a$ and these polynomials are identical to those appearing in the expansion \eqref{Hnasympt} (see Appendix \ref{appendixa1} for details).

To find an upper bound for the remainder $\widetilde{\ep}_p(N,\a)$, we need to estimate the integral
\begin{equation}\label{cpu+}
  c_p(u)=\int_{\G_0^+\cup\R^+}{z^{-1/2}\over w(z)^{2p}[w(z)-u]}dz,~~u\in\R.
\end{equation}
By expressing $w(z)$ in polar coordinates, one can easily show that
\begin{align*}
  |w(z)-u|\ge|\im w(z)|\ge{|\im w(z)^2|\over2|w(z)|}.
\end{align*}
Note from \eqref{w+-} that
\begin{align*}
  |\im w(z)^2|=|\im f(z)-\th_0/2|=\begin{cases}
    \th_0/2,&~z\in\R^+,\\
    (\pi-\th_0)/2,&~z\in\G_0^+.
  \end{cases}
\end{align*}
Assume that $p\geq 1$. Substituting the above two formulas into \eqref{cpu+} yields
\begin{align*}
  |c_p(u)|\le {4\over\th_0}\int_0^{+\infty} {s^{-1/2}\over|w(s)|^{2p-1}}ds+{4\over\pi-\th_0}\int_0^{+\infty} {s^{-1/2}\over|w(se^{i\pi})|^{2p-1}}ds.
\end{align*}
For $s\in(0,2)$, we have
\begin{align*}
  |w(s)|&\ge\sqrt{|\im w(s)^2|}=\sqrt{\th_0/2},\\
  |w(se^{i\pi})|&\ge \sqrt{|\im w(se^{i\pi})^2|}=\sqrt{(\pi-\th_0)/2},
\end{align*}
and for $s>2$, we have
\begin{align*}
  |w(s)|^2&\ge|\re w(s)^2|=s^2-2s\cos \a +{\ln s\over2}+{\cos^2\a +\ln 2+1/2\over2}\ge(s-1)^2,\\
  |w(se^{i\pi})|^2&\ge |\re w(se^{i\pi})^2|=s^2+2s\cos \a +{\ln s\over2}+{\cos^2\a +\ln 2+1/2\over2}\ge s^2.
\end{align*}
Accordingly,
\begin{align*}
  |c_p(u)| \le&\; {4\over\th_0}\left\{\int_0^2{s^{-1/2}\over(\th_0/2)^{p-1/2}}ds+\int_2^{+\infty}{s^{-1/2}\over (s-1)^{2p-1}}ds\right\}
  \notag\\& +{4\over\pi-\th_0}\left\{\int_0^2 {s^{-1/2}\over[(\pi-\th_0)/2]^{p-1/2}}ds+\int_2^{+\infty}{s^{-1/2}\over s^{2p-1}}ds\right\}
  \notag\\ \le&\; {4\over\th_0}\left\{{2^{3/2}\over(\th_0/2)^{p-1/2}}+{1\over2p-3/2}\right\}
  +{4\over\pi-\th_0}\left\{{2^{3/2}\over[(\pi-\th_0)/2]^{p-1/2}}+{2^{3/2-2p}\over2p-3/2}\right\}.
\end{align*}
Denote the right-hand side of the above inequality by $C_p$. It then follows from \eqref{int-k}, \eqref{epp+} and \eqref{cpu+} that
\begin{equation*}
  |\widetilde{\ep}_p(N,\a)|\le \frac{\sqrt{\sin\a}}{2\pi^{3/2}}{C_p\G(p+1/2)\over N^{p}},
\end{equation*}
provided $p \geq 1$. We define
\[
\ep_p(N,\a ) := \re(\widetilde{\ep}_p(N,\a)e^{-iN\theta_0/2}),
\]
and observe that
\begin{align*}
\left|\ep_p(N,\a)\right| & \le \left|\re\left(A_p(i\cot\a)e^{-i(\theta_0 N-\pi/2)/2}\right)\right|\frac{1}{N^p} + \left|\re\widetilde{\ep}_{p+1}(N,\a)\right| \\ & \le \left|\re\left(A_p(i\cot\a)e^{-i(\theta_0 N-\pi/2)/2}\right)\right|\frac{1}{N^p} + \left|\widetilde{\ep}_{p+1}(N,\a)\right|
\end{align*}
for any $p\geq 0$. Using the explicit value
\[
-f(t_+) =\frac{\cos(2\a)}{4}+\frac{1+\ln 2}{2} - i\frac{\theta_0}{2}
\]
in \eqref{Inexpansion}, and the fact that $H_n(\sqrt N \cos \a)=I_+ + I_- = I_+ +\overline{I_+}=2\re I_+$, we can summarize our results as follows.

\begin{thm}\label{thm-case2} Let $\a\in(0,\pi/2]$, $N=2n+1 \geq 1$ and $\theta_0 =\a -\sin(2\a)/2$. Then for any $p\ge 0$, we have
\begin{align*}
H_n(\sqrt N \cos \a)  =\frac{2^{n+1}n!e^{N(\cos(2\a )+2)/4}}{\sqrt{2\pi\sin\a} N^{(n + 1)/2}}&\left\{\sum\limits_{j=0}^{p-1} \frac{\re\left(A_j(i\cot\a)e^{-i(\theta_0 N-\pi/2)/2}\right)}{N^j}+\ep_p(N,\a) \right\} \\
  =\frac{2^{n+1}n!e^{N(\cos(2\a )+2)/4}}{\sqrt{2\pi\sin\a} N^{(n + 1)/2}}&\left\{\sum\limits_{j=0}^{p-1}\frac{\re(A_j(i\cot\a))}{N^j}\cos\left(\frac{\theta_0 N}{2}-\frac{\pi}{4}\right) \right.\\ & \left.+ \sum\limits_{j=1}^{p-1} \frac{\im(A_j (i\cot \alpha ))}{N^j}\sin\left(\frac{\theta_0 N}{2}-\frac{\pi}{4}\right)+\ep_p(N,\a) \right\},
\end{align*}
where
\begin{equation}\label{bound2}
  |\ep_p (N,\a)|\le {\tC_p\over N^p},
\end{equation}
with
\begin{equation*}
  \tC_p=\left|\re\left(A_p(i\cot\a)e^{-i(\theta_0 N-\pi/2)/2}\right)\right|+\frac{\sqrt{\sin\a}}{2\pi^{3/2}}{C_{p+1} \G(p+3/2)\over N},
\end{equation*}
and
\begin{align*}
C_{p+1}={4\over\th_0}\left\{{2^{3/2}\over(\th_0/2)^{p+1/2}}+{1\over2p+1/2}\right\}
  +{4\over\pi-\th_0}\left\{{2^{3/2}\over[(\pi-\th_0)/2]^{p+1/2}}+{2^{-1/2-2p}\over2p+1/2}\right\}.
\end{align*}
In particular, $\tC_p\to |\re(A_p(i\cot\a)e^{-i(\theta_0 N-\pi/2)/2})|$ as $N\to+\infty$; namely, the error bound is close to the absolute value of the first omitted term for large $N$. The coefficients $A_j(i\cot\a)$ are polynomials
in $i\cot\a$ of degree $3j$ with rational coefficients, and these polynomials can be calculated using a recursive formula given in Appendix \ref{appendixa1}.
\end{thm}

\section{Case III: \texorpdfstring{$x=1$}{x = 1}}\label{sec4}
The saddle points coincide $t_s=t_\pm=1/2$.
We have $f(t_s)=-3/4-(\ln 2)/2$, $f'(t_s)=f''(t_s)=0$ and $f'''(t_s)=8$.
We introduce the analytic function
\begin{equation}\label{w-pm-f}
  w(z)=[f(z)-f(t_s)]^{1/3},
\end{equation}
with the branch of the cube root being chosen so that
\begin{equation}\label{w-pm0}
  w(z)=(4/3)^{1/3}(z-t_s)e^{-2i\pi/3}+\mathcal{O}((z-t_s)^2)
\end{equation}
for $z$ in a small complex neighborhood of $t_s$.
We shall extend the definition of $w(z)$ by analytic continuation to $\C\cut(-\infty,0]$.
First, we investigate the equation $\im[f(z)-f(t_s)]=\im f(z)=0$, which in polar coordinates $z=re^{i\th}$ is
\begin{equation}\label{r-eq0}
  r^2\sin(2\th)-2r\sin\th+\th/2=0.
\end{equation}
This equation is clearly satisfied by all $z>0$. There are two further solutions:
\begin{align*}
  r_+(\th)=&\begin{cases}
    \ds{\th/(2\sin\th)\over1+\sqrt{1-\th\cot\th}},&~\th\in(0,\pi),\\\\
    \ds{1\over2},&~\th=0,\\\\
    \ds{1+\sqrt{1-\th\cot\th}\over2\cos\th},&~\th\in(-\pi/2,0),
  \end{cases}
\end{align*}
and
\begin{align*}
  r_-(\th)=&\begin{cases}
\ds{1+\sqrt{1-\th\cot\th}\over2\cos\th},&~\th\in(0,\pi/2),\\\\
    \ds{1\over2},&~\th=0,\\\\
\ds{\th/(2\sin\th)\over1+\sqrt{1-\th\cot\th}},&~\th\in(-\pi,0),
  \end{cases}
\end{align*}
where we have defined $r_\pm(\th)$ as piecewise functions so that they are differentiable at $\th=0$.
Note that since $1-\th\cot\th=\th^2/3+\mathcal{O}(\th^3)$ for small $|\th|>0$, we have
$$\sqrt{1-\th\cot\th}={|\th|\over\sqrt3}+\mathcal{O}(\th^2)$$
and
\begin{align*}
  r_\pm(\th)={1\over2}\mp{\th\over2\sqrt3}+\mathcal{O}(\th^2)
\end{align*}
for small values of $\th$. Especially,
\begin{align}\label{r-pm0}
  r_\pm(0)={1\over2},~~r_\pm'(0)=\mp{1\over2\sqrt3}.
\end{align}
Now, we introduce two contours
\begin{align*}
  \G_1 =\{r_+(\th)e^{i\th}:\th\in(-\pi/2,\pi)\},~~
  \G_2 =\{r_-(\th)e^{i\th}:\th\in(-\pi,\pi/2)\},
\end{align*}
which we illustrate in Figure \ref{fig-case3}.

\begin{figure}[htp]
  \centering
  \includegraphics[height=7cm]{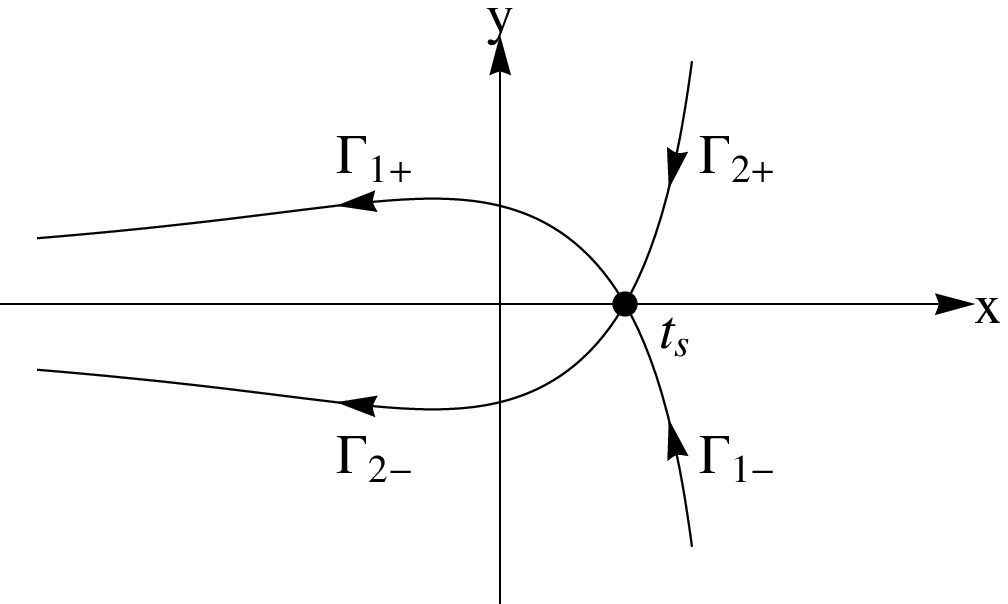}
  \caption{The contours $\G_1$ and $\G_2$. The arrows indicate the increasing direction of the phase function $f$.}
  \label{fig-case3}
\end{figure}
We have the following lemma.
\begin{lem}
  Let $z_+(\th)=r_+(\th)e^{i\th}$ for $\th\in(-\pi/2,\pi)$; the function $f(z_+(\th))-f(t_s)$ increases from $-\infty$ to $+\infty$ as $\th$ increases from $-\pi/2$ to $\pi$.
  Let $z_-(\th)=r_-(\th)e^{i\th}$ for $\th\in(-\pi,\pi/2)$; the function $f(z_-(\th))-f(t_s)$ decreases from $+\infty$ to $-\infty$ as $\th$ increases from $-\pi$ to $\pi/2$.
\end{lem}
\begin{proof}
  For small $\th$, we have from $f'(t_s)=f''(t_s)=0$, $f'''(t_s)=8$ and \eqref{r-pm0} that
  \begin{align*}
    z_+(\th)-t_s=[r_+'(0)+ir_+(0)]\th+\mathcal{O}(\th^2)={e^{2i\pi/3}\over\sqrt3}\th+\mathcal{O}(\th^2),
  \end{align*}
  and
  \begin{align*}
    f(z_+(\th))-f(t_s)={4\over3}(z_+(\th)-t_s)^3+\mathcal{O}((z_+(\th)-t_s)^4)
    ={4\over9\sqrt3}\th^3+\mathcal{O}(\th^4).
  \end{align*}
  Hence, the function $f(z_+(\th))-f(t_s)$ is increasing for small $\th$.
  We claim that $f(z_+(\th))-f(t_s)$ is increasing for all $\th\in(-\pi/2,\pi)$.
  Assume to the contrary that
  $$0={d\over d\th}f(z_+(\th))=f'(z_+(\th))[r_+'(\th)+ir_+(\th)]e^{i\th}$$
  for some $\th\in(-\pi/2,0)\cup(0,\pi)$. Since $f'(z)\neq0$ for any $z\neq t_s$, we have
  $r_+'(\th)=r_+(\th)=0$. This contradicts to the equation \eqref{r-eq0}. Hence, the function $f(z_+(\th))-f(t_s)$ is increasing for all $\th\in(-\pi/2,\pi)$.
  Note that $r_+(\th)\to\infty$ as $\th$ tends to $-\pi/2$ from the right or $\pi$ from the left.
  It is readily seen that $f(z_+(\th))-f(t_s)=r_+(\th)^2\cos(2\th)-2r_+(\th)\cos\th+(1/2)\ln r_+(\th)-f(t_s)$
  increases from $-\infty$ to $+\infty$ as $\th$ increases from $-\pi/2$ to $\pi$.

The second statement of the lemma follows in a similar manner.
\end{proof}
Since the solutions of $\im f(z)=0$ are the union $\G_1\cup\G_2\cup\R^+$, the equation $f(z)-f(t_s)=0$ has a unique solution $z=t_s$ in $\C\cut(-\infty,0]$. This implies that $w(z)$ defined in \eqref{w-pm-f} can be analytically continued to $\C\cut(-\infty,0]$.
Moreover, it follows from \eqref{w-pm0} that $w(z_+(\th))$ increases from $-\infty$ to $\infty$ as $\th$ increases from $-\pi/2$ to $\pi$.

Next, we let $\G_{1+}$ be the portion of $\G_1$ in the upper half-plane and $\G_{2-}$ be the portion of $\G_2$ in the lower half-plane; namely,
\begin{align*}
  \G_{1+}=\{r_+(\th)e^{i\th}:\th\in[0,\pi)\},~~
  \G_{2-}=\{r_-(\th)e^{i\th}:\th\in(-\pi,0]\}.
\end{align*}
Both $\G_{1+}$ and $\G_{2-}$ are oriented with increasing $\th$ (counter clockwise direction). We now deform the contour of integration for $H_n(\sqrt N)$ in \eqref{Hn-G} from $\G$ to $\G_{1+}\cup\G_{2-}$:
\[
H_n(\sqrt N)= {n!\over2\pi iN^{n/2}}\int_{\G_{1+}} e^{-Nf(t)}t^{-1/2}dt + {n!\over2\pi iN^{n/2}}\int_{\G_{2-}} e^{-Nf(t)}t^{-1/2}dt =:I_+ +I_-,
\]
where the last equality defines the two integrals $I_+$ and $I_-$, respectively. Since the contours $\G_{1+}$ and $\G_{2-}$ are symmetric with respect to the real axis, have opposite orientations, and $f(t) = \overline{f(\bar t)}$, the integrals $I_+$ and $I_-$ are complex conjugates of one another. Thus, it is sufficient to study the integral $I_+$. First, we introduce the change of variable $u = w(t)$:
\begin{equation}\label{Hn-G12}
I_+ = {n!e^{-Nf(t_s)}\over2\pi iN^{n/2}} \int_0^{+\infty} e^{-Nu^3}{[w^{-1}(u)]^{-1/2}\over w'(w^{-1}(u))}du.
\end{equation}
We denote by $\G_0^\pm$ the negative real lines with arguments $\pm\pi$ and orientation from left to right.
For convenience, we denote by $\G_*$ the union of $\G_0^+$ oriented to the right and $\G_0^-$ oriented to the left.
Since $w(z)$ is analytic in $\C\cut(-\infty,0]$ and $|w(z)|/|z|^{2/3}\to1$ as $z\to\infty$, we can assert from the Cauchy integral formula that
\begin{align*}
  {[w^{-1}(u)]^{-1/2}\over w'(w^{-1}(u))}
  ={1\over2\pi i}\int_{\G_*}{z^{-1/2}\over w(z)-u}dz.
\end{align*}
Substituting this into \eqref{Hn-G12} gives
\begin{align*}
  I_+={n!e^{-Nf(t_s)}\over(2\pi i)^2N^{n/2}}\int_0^{+\infty} e^{-Nu^3}
\int_{\G_*}{z^{-1/2}\over w(z)-u}dzdu.
\end{align*}
For any $p\ge 1$, we use the identities
\begin{equation*}
  {1\over w (z)-u}=\sum_{j=1}^{p-1}{u^{j-1}\over w (z)^j}+{u^{p-1}\over w (z)^{p-1}[w (z)-u]}
\end{equation*}
and
\begin{equation}\label{cubicintegral}
  \int_0^{+\infty} e^{-Nu^3}u^{j-1}du={\G(j/3)\over3N^{j/3}}
\end{equation}
to find
\begin{equation}\label{Inexpansion2}
I_+ = \frac{2^{1/2}n!e^{-Nf(t_s)}}{6\pi N^{n/2}}\left\{ \sum\limits_{j = 1}^{p-1} \left(\frac{3}{4}\right)^{j/3} D_j e^{2\pi ij/3-\pi i/2}\frac{\G(j/3)}{N^{j/3}} + \widetilde{\ep}_p(N) \right\},
\end{equation}
where
\begin{align}
\label{Dintegral}  D_j & = \left(\frac{4}{3}\right)^{j/3}\frac{e^{-2\pi ij/3}}{2^{3/2}\pi i}\int_{\G_*}{z^{-1/2}\over w(z)^j}dz,\\
 \label{remainder} \widetilde{\ep}_p(N) & = -\frac{3}{2^{3/2}\pi}\int_\R e^{-Nu^3}u^{p-1}\int_{\G_*}{z^{-1/2}\over w(z)^{p-1}[w(z)-u]}dzdu.
\end{align}
It is shown in Appendix \ref{appendixa3} that the coefficients $D_j$ are rational numbers which can be calculated via a recurrence relation.

To find an upper bound for the remainder $\widetilde{\ep}_p(N)$, we need to estimate the integral
\begin{equation}\label{cpu-pm}
  c_p(u)=\int_{\G_*}{z^{-1/2}\over w(z)^{p-1}[w(z)-u]}dz,~~u\in\R.
\end{equation}
By expressing $w(z)$ in polar coordinates, one can easily show that
\begin{align*}
  |w(z)-u|\ge|\im w(z)|\ge{|\im w(z)^3|\over3|w(z)|^2}.
\end{align*}
We also note from \eqref{w-pm-f} that
\begin{align*}
  |\im w(z)^3|=|\im f(z)|=\pi/2,
\end{align*}
for $z=se^{\pm i\pi}\in\G_*$.
Assume that $p\ge4$. Substituting the above two formulas into \eqref{cpu-pm} yields
\begin{align*}
  |c_p(u)|\le {6\over\pi}\int_0^{+\infty} {s^{-1/2}\over|w(se^{i\pi})|^{p-3}}ds+{6\over\pi}\int_0^{+\infty} {s^{-1/2}\over|w(se^{-i\pi})|^{p-3}}ds.
\end{align*}
For $z=se^{\pm i\pi}\in\G_*$ with $s=|z|\in(0,1)$, we have
\begin{align*}
  |w(z)|\ge|\im w(z)^3|^{1/3}=(\pi/2)^{1/3}.
\end{align*}
For $z=se^{\pm i\pi}\in\G_*$ with $s=|z|>1$, we have
\begin{align*}
  |w(z)|^3\ge|\re w(z)^3|=s^2+2s+{\ln s\over2}+{\ln 2+3/2\over2}\ge s^2.
\end{align*}
Consequently,
\begin{align*}
  |c_p (u)|\le {12\over\pi}\left\{\int_0^1{s^{-1/2}\over(\pi/2)^{p/3-1}}ds+\int_1^{+\infty}{s^{-1/2}\over s^{2p/3-2}}ds\right\}
  \le{12\over\pi}\left\{{2\over(\pi/2)^{p/3-1}}+{6\over4p-15}\right\}.
\end{align*}
Denote the right-hand side of the above inequality by $C_p$. It then follows from \eqref{cubicintegral}, \eqref{remainder} and \eqref{cpu-pm} that
\[
\left|\widetilde{\ep}_p(N)\right| \le \frac{1}{2^{3/2}\pi}\frac{C_p\Gamma(p/3)}{N^{p/3}},
\]
provided $p \geq 4$. We define
\[
\ep_p (N) := \re \widetilde{\ep}_p (N),
\]
and observe that
\begin{align*}
\left|\ep_p(N)\right| & \le \left|\left(\frac{3}{4}\right)^{p/3}D_p\sin\left(\frac{2\pi p}{3}\right)\right|\frac{\G(p/3)}{N^{p/3}} + \left|\re\widetilde{\ep}_{p+1}(N)\right| \\ & \le \left|\left(\frac{3}{4}\right)^{p/3} D_p \sin\left(\frac{2\pi p}{3}\right)\right|\frac{\G(p/3)}{N^{p/3}} + \left|\widetilde{\ep}_{p + 1}(N)\right|
\end{align*}
for any $p\ge 3$. Using the explicit value $-f(t_s) = 3/4+(\ln 2)/2$ in \eqref{Inexpansion2}, and the fact that $H_n(\sqrt N)=I_+ + I_- = I_+ +\overline{I_+}=2\re I_+$, we can summarize our results as follows.

\begin{thm}\label{thm-case3} Denote $N=2n+1 \geq 1$. Then for any $p\ge 3$, we have
\[
  H_n(\sqrt N)= \frac{2^{n + 1}n!e^{3N/4}}{3\pi N^{n/2}}\left\{\sum\limits_{j=1}^{p-1} \left(\frac{3}{4}\right)^{j/3} D_j \sin \left(\frac{2\pi j}{3}\right)\frac{\G(j/3)}{N^{j/3}} + \ep_p (N)\right\},
\]
where
\begin{equation}\label{bound3}
\left|\ep_p(N)\right| \le \frac{\tC_p}{N^{p/3}},
\end{equation}
with
\[
\tC_p  = \left|\left(\frac{3}{4}\right)^{p/3} D_p\sin\left(\frac{2\pi p}{3}\right)\right|\G(p/3)+\frac{1}{2^{3/2}\pi}\frac{C_{p+1}\G((p+1)/3)}{N^{1/3}},
\]
and
\begin{align*}
C_{p+1}={12\over\pi}\left\{{2\over(\pi/2)^{(p-2)/3}}+{6\over4p-11}\right\}.
\end{align*}
In particular, if $p$ is not divisible by $3$, $\tC_p \to \left|(3/4)^{p/3}D_p\sin\left(2\pi p/3\right)\right|\G(p/3)$ as $N \to +\infty$; namely, the error bound is close to the absolute value of the first neglected term when $N$ is large. The coefficients $D_j$ are rational numbers and can be calculated using a recurrence relation given in Appendix \ref{appendixa3}.
\end{thm}

\section{Numerical examples}\label{sec5}
In the three tables below, we present some numerical results that demonstrate the accuracy of our error bounds given in Theorems \ref{thm-case1}, \ref{thm-case2} and \ref{thm-case3}, respectively. We can infer from the tables that the bounds are rather realistic, that is, they do not seriously overestimate the actual error.\vspace{-4pt}

\begin{table*}[!htb]\small
\begin{center}
\begin{tabular}
[c]{ l r r}\hline
 & \\ [-2ex]
 values of $n$, $\b$ and $p$ & $n=50$, $\b=1$, $p=1$ & $n=50$, $\b=1$, $p=3$\\ [1ex]
 numerical value of $\left|\ep_p(N,\b)\right|$ & $0.0985 \times 10^{-2}$ & $0.0543 \times 10^{-5}$\\ [1ex]
 the bound \eqref{bound1} for $\left|\ep_p(N,\b)\right|$ & $0.1917 \times 10^{-2}$ & $0.1704 \times 10^{-5}$ \\ [1ex]\hline
 & \\ [-2ex]
 values of $n$, $\b$ and $p$ & $n=50$, $\b=4$, $p=1$ & $n=50$, $\b=4$, $p=3$\\ [1ex]
 numerical value of $\left|\ep_p(N,\b)\right|$ & $0.8229 \times 10^{-3}$ & $0.1879 \times 10^{-7}$\\ [1ex]
 the bound \eqref{bound1} for $\left|\ep_p(N,\b)\right|$ & $0.9811 \times 10^{-3}$ & $0.7375 \times 10^{-7}$ \\ [1ex]\hline
 & \\ [-2ex]
 values of $n$, $\b$ and $p$ & $n=100$, $\b=1$, $p=1$ & $n=100$, $\b=1$, $p=3$\\ [1ex]
 numerical value of $\left|\ep_p(N,\b)\right|$ & $0.5004 \times 10^{-3}$ & $0.0701 \times 10^{-6}$\\ [1ex]
 the bound \eqref{bound1} for $\left|\ep_p(N,\b)\right|$ & $0.7357 \times 10^{-3}$ & $0.1441 \times 10^{-6}$ \\ [1ex]\hline
 & \\ [-2ex]
 values of $n$, $\b$ and $p$ & $n=100$, $\b=4$, $p=1$ & $n=100$, $\b=4$, $p=3$\\ [1ex]
 numerical value of $\left|\ep_p(N,\b)\right|$ & $0.4134 \times 10^{-3}$ & $0.2383 \times 10^{-8}$\\ [1ex]
 the bound \eqref{bound1} for $\left|\ep_p(N,\b)\right|$ & $0.4533 \times 10^{-3}$ & $0.5887 \times 10^{-8}$ \\ [1ex]\hline
\end{tabular}
\end{center}
\setlength{\abovecaptionskip}{-2pt}\setlength{\belowcaptionskip}{5pt}
\caption{Bounds for $\left|\ep_p(N,\b)\right|$ with various $N=2n+1$, $\b$ and $p$, using \eqref{bound1}.}
\label{table1}
\end{table*}

\begin{table*}[!ht]\small
\begin{center}
\begin{tabular}
[c]{ l r r}\hline
 & \\ [-2ex]
 values of $n$, $\a$ and $p$ & $n=50$, $\a=\pi/4$, $p=1$ & $n=50$, $\a=\pi/4$, $p=3$\\ [1ex]
 numerical value of $\left|\ep_p(N,\a)\right|$ & $0.0406 \times 10^{-1}$ & $0.0103 \times 10^{-3}$\\ [1ex]
 the bound \eqref{bound2} for $\left|\ep_p(N,\a)\right|$ & $0.1147 \times 10^{-1}$ & $0.3159 \times 10^{-3}$ \\ [1ex]\hline
 & \\ [-2ex]
 values of $n$, $\a$ and $p$ & $n=50$, $\a=\pi/3$, $p=1$ & $n=50$, $\a=\pi/3$, $p=3$\\ [1ex]
 numerical value of $\left|\ep_p(N,\a)\right|$ & $0.1551 \times 10^{-2}$ & $0.0119 \times 10^{-4}$\\ [1ex]
 the bound \eqref{bound2} for $\left|\ep_p(N,\a)\right|$ & $0.2779 \times 10^{-2}$ & $0.1192 \times 10^{-4}$ \\ [1ex]\hline
 & \\ [-2ex]
 values of $n$, $\a$ and $p$ & $n=100$, $\a=\pi/4$, $p=1$ & $n=100$, $\a=\pi/4$, $p=3$\\ [1ex]
 numerical value of $\left|\ep_p(N,\a)\right|$ & $0.0689 \times 10^{-2}$ & $0.0066 \times 10^{-4}$\\ [1ex]
 the bound \eqref{bound2} for $\left|\ep_p(N,\a)\right|$ & $0.2494 \times 10^{-2}$ & $0.2001 \times 10^{-4}$ \\ [1ex]\hline
 & \\ [-2ex]
 values of $n$, $\a$ and $p$ & $n=100$, $\a=\pi/3$, $p=1$ & $n=100$, $\a=\pi/3$, $p=3$\\ [1ex]
 numerical value of $\left|\ep_p(N,\a)\right|$ & $0.0932 \times 10^{-2}$ & $0.1532 \times 10^{-6}$\\ [1ex]
 the bound \eqref{bound2} for $\left|\ep_p(N,\a)\right|$ & $0.1249 \times 10^{-2}$ & $0.8399 \times 10^{-6}$ \\ [1ex]\hline
\end{tabular}
\end{center}
\setlength{\abovecaptionskip}{-2pt}\setlength{\belowcaptionskip}{5pt}
\caption{Bounds for $\left|\ep_p(N,\a)\right|$ with various $N=2n+1$, $\a$ and $p$, using \eqref{bound2}.}
\label{table2}
\end{table*}

\begin{table*}[!htb]\small
\begin{center}
\begin{tabular}
[c]{ l r r r}\hline
 & \\ [-2ex]
 values of $n$ and $p$ & $n=50$, $p=4$ & $n=50$, $p=7$ & $n=50$, $p=10$\\ [1ex]
 numerical value of $\left|\ep_p(N)\right|$ & $0.2148 \times 10^{-3}$ & $0.0300 \times 10^{-5}$ & $0.0538 \times 10^{-7}$\\ [1ex]
 the bound \eqref{bound3} for $\left|\ep_p(N)\right|$ & $0.6615 \times 10^{-3}$ & $0.4092 \times 10^{-5}$ & $0.6667 \times 10^{-7}$\\ [1ex]\hline
 & \\ [-2ex]
 values of $n$ and $p$ & $n=100$, $p=4$ & $n=100$, $p=7$ & $n=100$, $p=10$\\ [1ex]
 numerical value of $\left|\ep_p(N)\right|$ & $0.0835 \times 10^{-3}$ & $0.0601 \times 10^{-6}$ & $0.0523 \times 10^{-8}$\\ [1ex]
 the bound \eqref{bound3} for $\left|\ep_p(N)\right|$ & $0.2253 \times 10^{-3}$ & $0.6658 \times 10^{-6}$ & $0.5438 \times 10^{-8}$\\ [1ex]\hline
\end{tabular}
\end{center}
\setlength{\abovecaptionskip}{-2pt}\setlength{\belowcaptionskip}{5pt}
\caption{Bounds for $\left|\ep_p(N)\right|$ with various $N=2n+1$ and $p$, using \eqref{bound3}.}
\label{table3}
\end{table*}

\section*{Acknowledgment}
We are very grateful to the anonymous referee for his/her careful reading and valuable suggestions which have helped to improve the presentation of this paper. The second author was supported by the JSPS Postdoctoral Research Fellowship No. P21020.

\appendix
\section{The computation of the coefficients}\label{appendixa}

\subsection{The coefficients \texorpdfstring{$A_j$}{Aj}}\label{appendixa1}

Our starting point is the integral representation \eqref{firstAint}. We first reverse the orientation of the path $\G_0 \cup \G_+$ and thereby remove the minus sign in front of the integral. Next, by an appeal to Cauchy's theorem, the contour of integration can be shrunk into a small, positively oriented circle around $t_-$. After performing the change of variable $u=ze^{\b}$, we arrive at
\begin{equation}\label{Aintegral3}
A_j (\coth\b) = \sqrt{\frac{e^{-2\b}-1}{\pi}}\frac{\G(j+1/2)}{2\pi i}\oint_{(1/2+)}\frac{(2u)^{-1/2}}{(w(ue^{-\b})^2)^{j+1/2}}du.
\end{equation}
The contour of integration is a small loop surrounding $1/2$ in the positive sense. We would like to obtain a recursive scheme for the evaluation of the integrals in \eqref{Aintegral3}. To this end, we employ a method of Lauwerier \cite{Lauwerier52ASR}. The method requires the power series expansions about $1/2$ of the functions appearing in the integrand of \eqref{Aintegral3}, namely
\[
w(ue^{-\b})^2 = \left(u-\tfrac{1}{2}\right)^2\left((e^{-2\b}-1)+\sum\limits_{k=1}^\infty\frac{(-2)^{k+1}}{k+2}\left(u-\tfrac{1}{2}\right)^k\right)
\]
and
\[
(2u)^{ - 1/2}=\sum\limits_{k=0}^\infty \frac{(-1)^k}{2^k}\binom{2k}{k}\left(u-\tfrac{1}{2}\right)^k,
\]
respectively. An application of Lauwerier's method \cite[Eqs. (8) and (9)]{Lauwerier52ASR} (with the slightly different notation $P_j (t):=(-1)^j q_j (-t/2)$) then yields
\begin{equation}\label{Aexpl}
A_j (\coth\b) = \frac{(-1)^j}{\sqrt \pi}\left(\frac{1+\coth\b}{2}\right)^j \int_0^{+\infty}e^{-s}s^{j-1/2}P_{2j}((1+\coth\b)s)ds,
\end{equation}
where $P_0(t)=1$ and
\begin{equation}\label{Prec}
P_j(t) = \frac{1}{2^j}\binom{2j}{j}-\sum\limits_{k=1}^j\frac{2^{k}}{k+2}\int_0^t P_{j-k}(s)ds
\end{equation}
for $j\geq 1$. It is readily seen from \eqref{Prec} that $P_j(t)$ is a polynomial in $t$ of degree $j$ with rational coefficients. With the substitution $t=\coth\b$, \eqref{Aexpl} reduces to the more compact form
\[
A_j (t) = \frac{(-1)^j}{\sqrt\pi}\left(\frac{1+t}{2}\right)^j\int_0^{+\infty}e^{-s}s^{j-1/2}P_{2j}((1 + t)s)ds.
\]
The $A_j(t)$ is the product of $(1+t)^j$ and a polynomial in $t$ of degree $2j$, and has rational coefficients. In particular, $A_0(t)=1$ and
\begin{align*}
& A_1(t)=\frac{(1+t)(1+5t-5t^2)}{24},\;\; A_2(t)=\frac{(1 + t)^2 (-143+298t+231t^2-770t^3+385t^4)}{1152},
\\ &
A_3 (t)=\frac{(1+t)^3(-6187-240549t+750468t^2-334565t^3-1021020t^4+1276275t^5-425425t^6)}{414720}.
\end{align*}

Consider now the right-hand side of \eqref{Aintegral}. We deform the contour of integration into a small, positively oriented circle around $t_+$, employ the identity $\sqrt{2\sin\a}=e^{-\pi i/4}e^{-i\a/2}\sqrt{e^{2i\a}-1}$, and perform the change of variable $u=ze^{-i\a}$, to obtain the equivalent form
\begin{equation}\label{Aintegral2}
\sqrt{\frac{e^{2i\a}-1}{\pi}}\frac{\G(j+1/2)}{2\pi i}\oint_{(1/2+)}\frac{(2u)^{-1/2}}{(w(ue^{i\a})^2)^{j+1/2}}du.
\end{equation}
Since $\cos \a=\cosh(-i\a)$, we see that \eqref{Aintegral3} is identical to \eqref{Aintegral2} when $\b$ is replaced by $-i\a$. This confirms that the right-hand side of \eqref{Aintegral} is indeed $A_j(\coth(-i\a))=A_j(i\cot \a)$.

\subsection{The coefficients \texorpdfstring{$D_j$}{Dj}}\label{appendixa3}

We proceed similarly as in the case of the coefficients $A_j$. By Cauchy's theorem, we can deform the contour of integration in \eqref{Dintegral} into a small, positively oriented loop contour around the saddle point $t_s=1/2$, to obtain
\begin{equation}\label{Dintegralrep}
D_j = \left(\frac{4}{3}\right)^{j/3} \frac{e^{-2\pi ij/3}}{2^{3/2}\pi i}\oint_{(1/2+)} \frac{z^{-1/2}}{w(z)^j}dz = \left(\frac{4}{3}\right)^{j/3} \frac{1}{2\pi i}\oint_{(1/2+)}\frac{(2z)^{-1/2}}{(w^3(z))^{j/3}}dz.
\end{equation}
The square and cube roots in the second integral are defined to be positive for positive real $z$ and are defined by continuity elsewhere. To apply Lauwerier's method, we expand the functions in the rightmost integral in \eqref{Dintegralrep} into Taylor series about $1/2$:
\[
w^3 (z) = \left(z -\tfrac{1}{2}\right)^3 \sum\limits_{k=0}^\infty\frac{(-2)^{k+2}}{k+3}\left(z-\tfrac{1}{2}\right)^k,\;\;
(2z)^{ - 1/2}=\sum\limits_{k=0}^\infty \frac{(-1)^k}{2^k}\binom{2k}{k}\left(z-\tfrac{1}{2}\right)^k.
\]
Lauwerier's method \cite[Eqs. (8) and (9)]{Lauwerier52ASR} (with the notation $Q_j(t) :=(-1)^{j-1} q_{j-1}(3t/4)$) then gives
\begin{equation}\label{Dformula}
D_j = \frac{(-1)^{j-1}}{\G(j/3)}\int_0^{+\infty}e^{-s}s^{j/3-1}Q_j (s)ds,
\end{equation}
where $Q_1(t)=1$ and
\[
Q_{j+1}(t) =\frac{1}{2^j}\binom{2j}{j} -3\sum\limits_{k=0}^{j-1}\frac{2^{k+1}}{k + 4}\int_0^t Q_{j-k} (s)ds
\]
for $j\geq 1$. It is easy to verify that $Q_j(t)$ is a polynomial in $t$ of degree $j-1$ with rational coefficients. When integrating $Q_j (s)$ term-by-term in \eqref{Dformula}, the factor $\G(j/3)$ in the denominator cancels, showing that $D_j$ is always a rational number. In particular,
\[
D_1=1,\;\;D_2  = 0,\;\;D_3=-\frac{3}{20},\;\;D_4=\frac{1}{6},\;\;D_5=-\frac{9}{70},\;\;D_6=\frac{3}{40},\;\; D_7=-\frac{199}{7200},\;\; D_8=-\frac{3}{700}.
\]

\end{document}